\documentclass{article}

\usepackage{graphicx} 
\usepackage{graphicx} 
\usepackage{a4wide}
\usepackage[T1]{fontenc}
\usepackage{a4wide}
\usepackage{lipsum}
\usepackage{mathtools}
\usepackage{amsthm}
\usepackage{amsfonts}
\usepackage{yfonts}
\usepackage{amssymb}
\usepackage{xcolor}
\usepackage{mathrsfs}
\usepackage{amsmath}
\usepackage{cite}
\usepackage{epsfig}
\usepackage{stmaryrd}
\usepackage[utf8]{inputenc} 
\usepackage[shortlabels]{enumitem}
\usepackage[T1]{fontenc}
\usepackage{caption}
\usepackage{soul}

\usepackage[utf8]{inputenc}
\usepackage[utf8]{inputenc}
\usepackage{fullpage}
\usepackage{amsthm}
\usepackage[pdftex,colorlinks,citecolor=blue]{hyperref}
\usepackage{hyperref} 
\usepackage{cleveref}
\usepackage{amsmath, amssymb, amsbsy}
\usepackage{algorithm}
\usepackage{algcompatible}
\usepackage{algpseudocode}
\usepackage{tikz,caption}
\usetikzlibrary{shapes,arrows,calc}

\allowdisplaybreaks

\crefname{figure}{figure}{}
\Crefname{figure}{Figure}{}

\usetikzlibrary{patterns,calc}

\newcommand{\beq}[1]{\begin{equation}\label{#1}}
\newcommand{\enq}[0]{\end{equation}}

\newtheorem{theorem}{Theorem}[section]

\newtheorem{lemma}[theorem]{Lemma}
\AddToHook{env/lemma/begin}{\crefalias{theorem}{lemma}}
\newtheorem*{claim*}{Claim}

\AddToHook{env/claim/begin}{\crefalias{theorem}{claim}}

\AddToHook{env/proposition/begin}{\crefalias{theorem}{proposition}}
\newtheorem{corollary}[theorem]{Corollary}
\AddToHook{env/corollary/begin}{\crefalias{theorem}{corollary}}

\AddToHook{env/fact/begin}{\crefalias{theorem}{fact}}
\newtheorem*{fact*}{Fact}

\theoremstyle{definition}

\AddToHook{env/defn/begin}{\crefalias{theorem}{definition}}

\AddToHook{env/remark/begin}{\crefalias{theorem}{remark}}
\newtheorem*{remark*}{Remark}

\definecolor{andrey}{rgb}{0.3, 0.65,,0.2}
\definecolor{trees}{rgb}{1, .7,.5}
\definecolor{vertex}{rgb}{.65, 0,.85}

\usepackage[margin=1in]{geometry} 

\title{A Note on the Asymptotic Least Density of Covering Codes in $[q]^n$}
\author{Andrey Shapiro\thanks{King's College London. E-mail: \tt{andrey.shapiro@kcl.ac.uk}}}
\date{}

\begin{document}
\maketitle

\begin{abstract}
    In this short note we revisit the upper bound of the asymptotic least density of covering codes of radius $R$ in $[q]^n$ established by Krivelevich, Sudakov, and Vu \cite{covering}. We show that by using a slightly different optimization in their core theorem we can obtain a constant factor improvement to their upper bound. 
\end{abstract}

\section{Introduction}

Let $[q]^n$ denote the set of all words of length $n$ over the finite alphabet $[q]:=\{1,2,\ldots,q\}$.
A covering code $K$ of $[q]^n$ of radius $R$ is a set of points such that for all $z\in[q]^n$, there is some $z'\in K$ such that $z$ and $z'$ are Hamming distance at most $R$ apart. That is, $z$ and $z'$ differ on at most $R$ indices. Finding small covering codes is a central problem in coding theory and we refer the reader to the work by Cohen, Honkala, Litsyn, and Lobstein \cite{textbook} for a detailed exposition of covering codes.
The purpose of this small note is to point out a constant factor improvement of the result in the paper by Krivelevich, Sudakov, and Vu \cite{covering}. At the time, \cite{covering} established the best known upper bound of the asymptotic least density of covering codes of radius $R$ on $[q]^n$ and to the best of our knowledge, it has remained unimproved nearly 25 years later. Our improvement holds for $R\geq 6$ and comes from a minor alteration to their main theorem.

Let us begin by reviewing the necessary definitions. First, denote the size of a Hamming ball of radius $R$ in $[q]^n$ by $$V_q(n,R)\coloneqq\sum_{i=0}^R (q-1)^i{n\choose i},$$ and for $K$ a covering code of radius $R$ we define $|K|\cdot V_q(n,R)/q^n$ to be the density of $K$. Further, we define $\mu_q(n,R)$ to be the minimal density of a covering code of $[q]^n$ of radius $R$. Further, let $$\mu_q^*(R)\coloneqq \limsup_{n\to \infty} \mu_q(n,R).$$
Finally, for $A\subset [q]^n$ and $B\subset [q]^m$ define $A\oplus B\coloneqq\{(a,b):a\in A, b\in B\}\subset [q]^{n+m}.$

We now present our main theorem and corollary (note that here and for the entirety of this note all logarithms have natural base).
\begin{theorem}\label{mainT}
For all values of $q$, $R\geq 1$, $y>1$, and $x>0$ satisfying $e^{-x}y^R<1$ we have:
    $$\mu_q^*(R)\leq x\left(\frac{y}{y-1}\right)^R\left(1+\frac{1}{e^{x}y^{-R}-1}\right)$$
\end{theorem}

\begin{corollary}\label{mainC}
    For $R\geq 6$ and any value of $q$, $$\mu_q^*(R)\leq e^{(1.8+\log\log R)/\log R}R\log R.$$
\end{corollary}

For reference, we provide here the main theorem and corollary in \cite{covering}.

\begin{theorem}[Theorem 1.5 in \cite{covering}]\label{their_mainT}

For all values of $q$, $R>R_1\geq 1$, $y>1$, and $x>0$ satisfying $e^{-x}y^R<1$ we have:
    $$\mu_q^*(R)\leq x{R\choose R_1}^{-1}y^{R_1}\left(\frac{y}{y-1}\right)^{R-R_1}\left(1+\frac{1}{e^{x}y^{-R}-1}\right)\mu_q^*(R_1)$$
    
\end{theorem}

\begin{corollary}[Corollaries 1.4 and 1.6 in \cite{covering}]\label{their_mainC}
    For $R\geq 3$, $$\mu_2^*(R)\leq e(R\log R+\log R+\log\log R+2)$$ and for any $q\geq 3$ $$\mu_q^*(R)\leq 2e(R\log R+\log R+\log\log R+2).$$
\end{corollary}

Formally, \Cref{mainT} is what one obtains from Theorem \Cref{their_mainT} by setting $R_1=0$, since $\mu^*_q(0)=1$.
As a result, our proof will closely follow the proof of \Cref{their_mainT} presented in \cite{covering} while showing that it remains valid when setting $R_1=0$. We point this out because it indicates that the bound improves when we sidestep the dependency on $\mu_q^*(R_1)$ in \Cref{their_mainT}. As for \Cref{mainC}, it is easy to see it is stronger than \Cref{their_mainC} when $R\geq 6$. It should be noted that a similar version of \Cref{mainC} can be proven for $R\leq 5$ but it would ultimately yield a bound weaker than that of \Cref{their_mainC} when $q=2$. 

\section{Proofs of main results}

Since it is the more direct result, we will begin by proving \Cref{mainC} using \Cref{mainT}.
\begin{proof}[Proof of \Cref{mainC}]
    Set $y=R\log R +1$ and
    $x=R\log (y) + 2\log R$. Recall that $(1+\varepsilon)\leq e^\varepsilon$ and so $\left(\frac{y}{y-1}\right)^{R}\leq e^{R/(y-1)}=e^{1/\log R}.$ Next, observe that $y>1$, $x>0$, and $e^{-x}y^{R}=e^{-x}e^{R\log y}=R^{-2}<1$  and so we can apply \Cref{mainT} to get:
    \begin{align*}
         \mu_q^*(R)& \leq \left(\frac{y}{y-1}\right)^R(R\log (y) + 2\log R)\left(1+\frac{1}{R^2-1}\right)
         \\&\leq e^{1/\log R}(R\log (y) + 2\log R)\left(1+\frac{1}{R^2-1}\right).
    \end{align*}
Further, one can verify that $(R\log(R\log R+1)+2\log R)\left(1+\frac{1}{R^2-1}\right)<R(\log R + \log\log R+0.8)$ for $R\geq 6$. Then again using the fact that $(1+\varepsilon )\leq e^\varepsilon$ we get: 
    \begin{align*}
    \mu_q^*(R)
&< e^{1/\log R}R(\log R + \log\log R+0.8)\\&=e^{1/\log R}R\log R \left(1+\frac{\log\log R + 0.8}{\log R}\right)
         \\& \leq e^{(1.8+\log\log R)/\log R}R\log R.
    \end{align*}
\end{proof}

Now we turn to proving \Cref{mainT}. We begin by stating the following lemmas, the short proofs of which can be found in the original paper \cite{covering}. 

\begin{lemma}\label{limit}
Let $y>1$ be a constant and let $(a_n)$, $(b_n)$, and $(s_n)$ be sequences of positive numbers satisfying:
$$\limsup_{n\to \infty} a_n\leq a, \quad \quad \limsup_{n\to \infty} b_n\leq b<1, \quad \quad \text{and} \quad \quad s_n\leq a_n+b_ns_{\lfloor n/y\rfloor}.$$ Then $$\limsup_{n\to \infty} s_n\leq \frac{a}{1-b}.$$
\end{lemma}

\begin{lemma}\label{X_gen}

Let $x>0$ and let $G$ be a $d$-regular graph with $m$ vertices. Then there is a set $X$ of size at most $xm/(d+1)$ so that $$|\overline{N}(X)|\leq e^{-x+(d+1)/m}m$$ where $\overline{N}(X)\coloneqq V(G)\setminus (X\cup N(X)).$
    
\end{lemma}

With these lemmas, we proceed to the proof of \Cref{mainT}.

\begin{proof}[Proof of \Cref{mainT}]

Let us define $r=\lfloor n/y \rfloor$ and $r'=n-r$. Further, let us view $[q]^{r'}$ as a graph with vertices connected if they are of Hamming distance at most $R$ apart. This graph has $m=q^{r'}$ vertices and is $d$-regular with $d=V_q(r',R)-1$. Next we apply \Cref{X_gen} on this graph to get a set $X$ of size at most $xq^{r'}/V_q(r',R)$ such that $|\overline{N}(X)|\leq e^{-x+V_q(r',R)/q^{r'}}\cdot q^{r'}$. Next, taking $K_2$ to be an optimal covering code of $[q]^{r}$ with radius $R$, let us consider the covering code $$K=X \oplus [q]^{r} \cup \overline{N}(X)\oplus K_2.$$  We can check that it is indeed a covering code by observing that for each $(z_1,z_2)\in [q]^{r'} \times [q]^{r}$, either $z_1$ is distance at most $R$ from some vertex in $X$ and $z_2\in [q]^{r}$, or otherwise, $z_1\in \overline{N}(X)$ and $z_2$ is distance at most $R$ from some point in $K_2$. 

Now let us consider the size of $K$. First,
\begin{equation}\label{X_size}
    |X\oplus [q]^r|=|X|\cdot q^{r}\leq \frac{xq^{r'}}{V_q(r',R)}\cdot q^{r}=\frac{xq^{n}}{V_q(r',R)}.
\end{equation}
Second, since $K_2$ is an optimal covering, 
\begin{equation}\label{N_size}
    |\overline{N}(X)\oplus K_2|=|\overline{N}(X)|\cdot \mu_q(r,R)\cdot \frac{q^{r}}{V_q(r,R)}\leq e^{-x+V_q(r',R)/q^{r'}}\cdot \mu_q(r,R)\cdot \frac{q^{n}}{V_q(r,R)}.
\end{equation}
Then by \cref{X_size} and \cref{N_size}, because $K$ is a covering code of $[q]^n$ of radius $R$ we get the following bound on the optimal density $\mu_q(n,R)$:
$$\mu_q(n,R)\leq |K|\cdot \frac{V_q(n,R)}{q^n}\leq \frac{x\cdot V_q(n,R)}{V_q(r',R)} + e^{-x+V_q(r',R)/q^{r'}}\cdot \mu_q(r,R)\cdot \frac{V_q(n,R)}{V_q(r,R)}.$$

Finally, we set $$a_n\coloneqq \frac{x\cdot V_q(n,R)}{V_q(r',R)}, \quad \quad b_n\coloneqq e^{-x+V_q(r',R)/q^{r'}}\cdot \frac{V_q(n,R)}{V_q(r,R)}, \quad \quad \text{and} \quad \quad s_n\coloneqq \mu_q(n,R)$$ 
with the aim of applying \Cref{limit}. Let us begin by recalling that $V_q(n,R)=\sum_{i=0}^R (q-1)^i{n \choose i}\approx \frac{(n(q-1))^R}{R!}$, or rather, more precisely, $$\frac{V_q(n,R)}{V_q(r,R)}=(1+o(1))\frac{((q-1)n)^R}{((q-1){r})^R}\cdot \frac{R!}{R!}=(1+o(1))y^R \quad\quad  \text{and likewise} \quad\quad \frac{V_q(n,R)}{V_q(r',R)}=(1+o(1))\left(\frac{y}{y-1}\right)^R.$$ Further, $$\lim_{n\to \infty} e^{V_q(r',R)/q^{r'}}=1,$$
and so we can conclude that  $$a\coloneqq\limsup_{n\to\infty}a_n=\limsup_{n\to\infty} \frac{x\cdot V_q(n,R)}{V_q(r',R)}  =x\cdot \left(\frac{y}{y-1}\right)^R$$ and $$ b\coloneqq \limsup_{n\to\infty}b_n=\limsup_{n\to\infty}e^{-x+V_q(r',R)/q^{r'}}\cdot \frac{V_q(n,R)}{V_q(r,R)}=e^{-x}y^R.$$ 
So by \Cref{limit}, since we have ensured that $e^{-x}y^R<1$, we have $$\mu_q^*(R)=\limsup_{n\to \infty} s_n\leq \frac{a}{1-b} =\frac{x\cdot \left(\frac{y}{y-1}\right)^R}{1-e^{-x}y^R}$$ as desired.
\end{proof}

\subsection*{Acknowledgments}
The author wishes to thank Matthew Jenssen for helpful comments. 
\bibliographystyle{plain}
\bibliography{main}

@article {covering,
    AUTHOR = { Krivelevich, M. and Sudakov, B. and Vu, V.},
     TITLE = {Covering Codes With Improved Density},
   JOURNAL = { IEEE TRANSACTIONS ON INFORMATION THEORY},
    VOLUME = {49},
      YEAR = {2003},
    NUMBER = {7},
     PAGES = {1812--1815},
}

@book{textbook,
  title={Covering Codes},
  author={G{\'e}rard D. Cohen and Iiro S. Honkala and Simon Litsyn and Antoine Lobstein},
  booktitle={North-Holland Mathematical Library},
  year={2005},
  url={https://api.semanticscholar.org/CorpusID:195891379}
}

\end{document}